\documentclass[11pt,reqno]{article}
\usepackage{amsmath,amssymb,amsthm,amsfonts,amstext,amsbsy,amscd}
\usepackage[utf8]{inputenc}
\usepackage{bbm,dsfont}
\usepackage{color}
\usepackage{comment}

\usepackage[colorlinks,citecolor=blue,urlcolor=blue,linkcolor = blue,hypertexnames=false]{hyperref}

\theoremstyle{plain}
\newtheorem{thm}{Theorem}
\newtheorem{proposition}{Proposition}
\newtheorem{corollary}{Corollary}
\newtheorem{lemma}{Lemma}
\newtheorem{assumption}{Assumption}

\theoremstyle{definition}

\theoremstyle{remark}

\DeclareMathOperator{\R}{{\mathbb R}}

\DeclareMathOperator{\tr}{tr}

\DeclareMathOperator{\argmin}{arg\,min}

\newcommand{\SGN}{C_{\psi_2}}
\newcommand{\CEV}{C_{ev}}
\newcommand{\EstPCR}{\hat f_d^{PCR}}
\newcommand{\EstTPCR}{\tilde f_d^{PCR}}
\newcommand{\EstORA}{\hat f_d^{oracle}}
\newcommand{\EstTORA}{\tilde f_d^{oracle}}

\renewcommand{\phi}{\varphi}

\renewcommand{\le}{\leqslant}
\renewcommand{\ge}{\geqslant}

\numberwithin{equation}{section}

\allowdisplaybreaks[4]

\begin{document}
\title{A note on the prediction error of principal component regression}
\author{Martin Wahl\thanks{Institut f\"{u}r Mathematik, Humboldt-Universit\"{a}t zu Berlin, Unter den Linden 6, 10099 Berlin, Germany. E-mail: martin.wahl@math.hu-berlin.de 
\newline
\textit{2010 Mathematics Subject Classifcation.} 62H25\newline
}}
\date{}

\maketitle

\begin{abstract}
We analyse the prediction error of principal component regression (PCR) and prove non-asymptotic upper bounds for the corresponding squared risk. Under mild assumptions, we show that PCR performs as well as the oracle method obtained by replacing empirical principal components by their population counterparts. Our approach relies on upper bounds for the excess risk of principal component analysis. 
\end{abstract}

\section{Introduction}
\subsection{Model} Let $(\mathcal{H},\langle\cdot,\cdot\rangle)$ be a separable Hilbert space and let $(Y,X)$ be a pair of random variables satisfying the regression equation
\begin{equation}\label{EqModel}
Y=\langle f,X\rangle+\epsilon,
\end{equation}
where $f\in\mathcal{H}$, $X$ is a random variable with values in $\mathcal{H}$ such that $\mathbb{E}X=0$ and $\mathbb{E}\|X\|^2<\infty$, and $\epsilon$ is a real-valued random variable such that $\mathbb{E}(\epsilon|X)=0$ and $\mathbb{E}(\epsilon^2|X)=\sigma^2$. We suppose that we observe $n$ independent copies $(Y_1,X_1),\dots,(Y_n,X_n)$ of $(Y,X)$ and consider the problem of estimating $f$.

Allowing for general Hilbert spaces $\mathcal{H}$, the statistical model \eqref{EqModel} covers several regression problems. It includes a special case of the functional linear model, in which the responses are scalars and the covariates are curves. This model has been extensively studied in the literature, see e.g. the monographs by Horv\'{a}th and Kokoszka \cite{HK12} and Hsing and Eubank \cite{HE15}. It also includes several kernel learning problems. For instance, nonparametric regression with random design can be written in the form \eqref{EqModel}, provided that the regression function is contained in a reproducing kernel Hilbert space. This connection between statistical learning theory and the theory of linear inverse problems has been developed e.g. in De Vito et al.~\cite{VRC05}. 

In this note, we focus on estimating $f$ by principal component regression (PCR). PCR is a widely known estimation procedure in which principal component analysis (PCA) is applied in a first step to reduce the high dimensionality of the data. Then, in a second step, an estimator of $f$ is obtained by regressing the responses on the leading empirical principal components. Our main goal is to bound the prediction error of PCR using the reconstruction error of PCA. In particular, we show that the error due to regressing on empirical principal components instead of regressing on their population counterparts can be controlled by the excess risk of PCA. The situation is quite different in the case of error bounds in $\mathcal{H}$~norm, for which stronger perturbation bounds related to subspace distances are required.


A general account on PCR is given in the monograph by Jolliffe \cite{J02}. For a study of the prediction error of PCR, with focus on minimax optimal rates of convergence under standard model assumptions, see Hall and Horowitz \cite{HH07} and Brunel, Mas, and Roche \cite{BMR16} for the functional data context and Blanchard and M\"{u}cke \cite{BM18} for the statistical learning context. Note that the latter studies PCR (resp. spectral cut-off) within a larger class of spectral regularisation methods. Bounds for the reconstruction error using the machinery of empirical risk minimisation are derived in Blanchard, Bousquet, and Zwald~\cite{BBZ07}. The final result obtained in this paper rely on the perturbation approach to the excess risk developed in Rei\ss{} and Wahl \cite{RW17}.

\subsubsection*{Further notation} Let $\|\cdot\|$ denote the norm on $\mathcal{H}$.  Given a bounded (resp. Hilbert-Schmidt) operator $A$ on $\mathcal{H}$, we denote the operator norm (resp. the  Hilbert-Schmidt norm) of $A$ by $\|A\|_\infty$ (resp. $\|A\|_2$). Given a trace class operator $A$ on $\mathcal{H}$, we denote the trace of $A$ by $\operatorname{tr}(A)$. For $g\in\mathcal{H}$, the empirical and the $L^2(\mathbb{P}^X)$ norm are defined by $\|g\|_{n}^2=(1/n)\sum_{i=1}^n\langle g,X_i\rangle^2$ and $\|g\|_{L^2(\mathbb{P}^X)}^2=\mathbb{E}\langle g,X\rangle^2$, respectively. For $g,h\in\mathcal{H}$ we denote by $g\otimes h$ the rank-one operator defined by $(g\otimes h)x=\langle h,x\rangle g$, $x\in \mathcal{H}$. We write $\mathbf{Y}=(Y_1,\dots,Y_n)^T$ and $\boldsymbol{\epsilon}=(\epsilon_1,\dots,\epsilon_n)^T$. Throughout the paper, $c$ and $C$ denote constants that may change from line to line (by a numerical value).

\subsection{Principal component analysis}\label{SecPCA}
The covariance operator of $X$ is denoted by $\Sigma$. Since $X$ is assumed to be centered and strongly square-integrable, the covariance operator $\Sigma$ can be defined by $\Sigma=\mathbb{E} X\otimes X$, where expectation is taken in the Hilbert space of all Hilbert-Schmidt operators on $\mathcal{H}$. It is well-known that under the above assumptions, the covariance operator $\Sigma$ is a positive, self-adjoint trace class operator, see e.g. \cite[Theorem 7.2.5]{HE15}. By the spectral theorem, there exists a sequence $\lambda_1\ge \lambda_2\ge\dots>0$ of  positive eigenvalues (which is either finite or converges to zero) together with an orthonormal system of eigenvectors $u_1,u_2,\dots$ such that $\Sigma$ has the spectral decomposition
\begin{equation}\label{EqSpecDec}
\Sigma=\sum_{j\ge 1}\lambda_j P_j,\qquad P_j=u_j\otimes u_j.
\end{equation}
Without loss of generality we shall assume that the eigenvectors $u_1,u_2,\dots$ form an orthonormal basis of $\mathcal{H}$ such that $\sum_{j\ge 1}P_j=I$. Moreover, let
\begin{equation*}
\hat{\Sigma}=\frac{1}{n}\sum_{i=1}^nX_i\otimes X_i
\end{equation*}
be the sample covariance operator. Again, there exists a sequence $\hat{\lambda}_1\ge \hat{\lambda}_2\ge\dots\ge 0$ of eigenvalues together with an orthonormal basis of eigenvectors $\hat{u}_1,\hat{u}_2,\dots$ such that we can write
\begin{equation*}
\hat{\Sigma}=\sum_{j\ge 1}\hat{\lambda}_j \hat{P}_j,\qquad \hat{P}_j=\hat{u}_j\otimes\hat{u}_j.
\end{equation*}
For $d\ge 1$, we write
\[
P_{\le d}=\sum_{j\le d}P_j,\quad P_{>d}=\sum_{k>d}P_k, \qquad \hat{P}_{\le d}=\sum_{j\le d}\hat{P}_j,\quad \hat{P}_{> d}=\sum_{k> d}\hat{P}_k
\]
for the orthogonal projections onto the linear subspace spanned by the first $d$ eigenvectors of $\Sigma$ (resp. $\hat \Sigma$), and onto its orthogonal complement. Introducing ${\mathcal P}_{d}=\{P:{\mathcal H}\to{\mathcal H}\,|\, P\text{ is orthogonal projection of rank }d\},$
the reconstruction error of $P\in{\mathcal P}_{d}$ is defined by
\begin{equation*}
R(P)=\mathbb{E}[ \|X-PX\|^2].
\end{equation*}
The fundamental idea behind PCA is that $P_{\le d}$ satisfies
\begin{equation*}
P_{\le d}\in\argmin_{P\in {\mathcal P}_{d}}\limits R(P).
\end{equation*}
Similarly, the empirical reconstruction error of $P\in{\mathcal P}_{d}$ is defined by
\begin{equation*}
R_n(P)=\frac{1}{n}\sum_{i=1}^n \|X_i-PX_i\|^2,
\end{equation*}
and we have
\begin{equation*}
\hat{P}_{\le d}\in\argmin_{P\in {\mathcal P}_{d}}\limits R_n(P).
\end{equation*}
The excess risk of the PCA projector $\hat{P}_{\le d}$ is defined by
\begin{equation*}
{\mathcal E}^{PCA}_d=R(\hat{P}_{\le d})-\min_{P\in {\mathcal P}_{d}}R(P)=R(\hat{P}_{\le d})-R( P_{\le d}).
\end{equation*}
The excess risk is a natural measure of accuracy of the PCA projector $\hat P_{\le d}$. In Rei\ss{} and Wahl \cite{RW17} it is argued that the excess risk is more adequate for many tasks like reconstruction and prediction than e.g. the Hilbert-Schmidt distance $\|P_{\le d}- \hat{P}_{\le d}\|_2^2$. The main goal of this note is to answer this in the affirmative in the case of the prediction error of PCR.

\subsection{Principal component regression}\label{SecPCR}
Let $\hat U_d=\operatorname{span}(\hat u_1,\dots,\hat u_d)$ be the linear subspace spanned by the first $d$ eigenvectors of $\hat \Sigma$. Then the PCR estimator in dimension $d$ is defined by
\begin{equation*}
\EstPCR\in\argmin_{g\in \hat U_d}\limits\|\mathbf{Y}-S_ng\|_n^2
\end{equation*}
with sampling operator $S_n:\mathcal{H}\rightarrow \mathbb{R}^n,h\mapsto (\langle h,X_i\rangle)_{i=1}^n$.
The PCR estimator can be defined explicitly by using the singular value decomposition of $S_n$. In fact, writing
\[
n^{-1/2}S_n=\sum_{j=1}^n\hat\lambda_j^{1/2}\hat v_j\otimes\hat u_j
\]
with orthonormal basis $\hat v_1,\dots, \hat v_n$ of $\mathbb{R}^n$ with respect to the Euclidean inner product on $\mathbb{R}^n$, also denoted by $\langle\cdot,\cdot\rangle$ in what follows (use that $S_n^*S_n=n\hat\Sigma$ with adjoint operator $S_n^*$), we have
\begin{equation}\label{EqRepPCREst}
\EstPCR=n^{-1/2}\sum_{j\le d}\hat\lambda_j^{-1/2}\langle \mathbf{Y},\hat{v}_j\rangle \hat u_j,
\end{equation}
provided that $\hat \lambda_d>0$. Our main goal is to analyse the squared prediction error of PCR given by
\begin{equation}\label{EqPredLoss}
\|\EstPCR-f\|_{L^2(\mathbb{P}^X)}^2=\langle \EstPCR-f,\Sigma(\EstPCR-f)\rangle.
\end{equation}
To avoid technical issues related to the mean squared prediction error of PCR, we define our final estimator by 
\[
\EstTPCR=\EstPCR\text{ if }\hat\lambda_ d\ge \lambda_d/2\text{ and }\EstTPCR= 0\text{ otherwise}.
\]
Note that it is possible to replace $\lambda_d/2$ by a feasible threshold.

\subsection{The oracle estimator}
Letting $U_d=\operatorname{span}(u_1,\dots,u_d)$, we define the oracle PCR estimator in dimension $d$ by
\[
\EstORA\in\argmin_{g\in  U_d}\limits\|\mathbf{Y}-S_ng\|_n^2.
\] 
The oracle PCR estimator has an analogous decomposition as in \eqref{EqRepPCREst}, with empirical eigenvalues and singular vectors replaced by those corresponding to the projected observations $P_{\le d}X_i$, cf. \eqref{EqDecOracleEst} below. We define our final estimator by $\EstTORA=\EstORA$ if the $d$th largest eigenvalue of the empirical covariance operator  of the $P_{\le d}X_i$ is larger than or equal to $\lambda_d/2$ and $\EstTORA= 0$ otherwise. Using this modification, we prove in Appendix \ref{PropOracleRiskProof} the following standard risk bound, see e.g. \cite[Chapter 11]{GKKW02} for similar results for the linear least squares estimator in random design.
\begin{proposition}\label{PropOracleRisk} Suppose that $X$ is sub-Gaussian with factor $\SGN>0$, as defined in Assumption \ref{SubGauss} below. Then there are constants $c,C>0$ depending only on $\SGN$ such that 
\[
\forall d\le cn,\quad\mathbb{E}\|\EstTORA-f\|_{L^2(\mathbb{P}^X)}^2\le C\Big(\sum_{k>d}\lambda_k\langle f,u_k\rangle^2+\frac{\sigma^2d}{n}+R\Big)
\]
with $R=\lambda_{1}\lambda_{d}^{-1}(\sigma^2+\|f\|_{L^2(\mathbb{P}^X)}^2)e^{-cn}$.
\end{proposition}
 The oracle bias term $\sum_{k>d}\lambda_k\langle f,u_k\rangle^2$ is affected by the fact that $U_d$ does not necessarily have good approximation properties for $f$. In the extreme case, it is equal to $\lambda_{d+1}\|f\|^2$. Under additional assumptions, relating $f$ to the spectral characteristics of $\Sigma$, the latter can be further improved. In this note, we will restrict ourselves to a H\"older-type source condition, in which case the following corollary will serve as a benchmark later on:
\begin{corollary}\label{CorOracleRisk} Grant the assumptions of Proposition \ref{PropOracleRisk}. Suppose that $f=\Sigma^sh$ with $s\ge 0$ and $h\in\mathcal{H}$. Then we have 
\[
\forall d\le cn,\quad\mathbb{E}\|\EstTORA-f\|_{L^2(\mathbb{P}^X)}^2\le C\Big(\lambda_{d+1}^{1+2s}\|h\|^2+\frac{\sigma^2d}{n}+R\Big).
\]
\end{corollary}

\section{Analysis of the prediction error}

\subsection{The bias-variance decomposition} We start with deriving a bias-variance decomposition of the mean squared prediction error of PCR conditional on the design:
\begin{lemma}\label{LemBVDecCond} If $\hat\lambda_ d> 0$, then we have
\begin{align*}
\mathbb{E}\big(\|\EstPCR-f\|_{L^2(\mathbb{P}^X)}^2|X_1,\dots,X_n\big)=\langle \hat P_{>d}f,\Sigma \hat P_{>d}f\rangle+\frac{\sigma^2}{n}\sum_{j\le d}\hat\lambda_j^{-1}\operatorname{tr}( \hat P_j \Sigma).
\end{align*}
\end{lemma}
\begin{proof}
Suppose that $\hat \lambda_d>0$. Inserting $\mathbf{Y}=S_nf+\boldsymbol{\epsilon}$ into \eqref{EqRepPCREst}, we have
\begin{equation}\label{EqDecEstPCR}
\EstPCR=\hat P_{\le d}f+n^{-1/2}\sum_{j\le d}\hat\lambda_j^{-1/2}\langle \boldsymbol{\epsilon},\hat{v}_j\rangle \hat u_j.
\end{equation}
Inserting \eqref{EqDecEstPCR} and the identity $f=\hat P_{\le d}f+\hat P_{> d}f$ into \eqref{EqPredLoss}, we get
\begin{align*}
\|\EstPCR-f\|_{L^2(\mathbb{P}^X)}^2&=\langle \hat P_{>d}f,\Sigma \hat P_{>d}f\rangle
-2n^{-1/2}\sum_{j\le d}\hat\lambda_j^{-1/2}\langle \hat P_{> d}f, \Sigma \hat u_j\rangle\langle \boldsymbol{\epsilon},\hat{v}_j\rangle\\
&+n^{-1}\sum_{j\le d}\sum_{k\le d}\hat\lambda_j^{-1/2}\hat\lambda_k^{-1/2}\langle\hat u_j,\Sigma \hat u_k\rangle\langle \boldsymbol{\epsilon},\hat{v}_j\rangle\langle \boldsymbol{\epsilon},\hat{v}_k\rangle.
\end{align*}
The result now follows from the fact that, conditional on the design, the $\langle \boldsymbol{\epsilon}, \hat v_j\rangle$ are uncorrelated, each with expectation zero and variance $\sigma^2$.
\end{proof}

\subsection{Relating the bias to the excess risk} The following lemma shows a clear connection of the bias term in Lemma  \ref{LemBVDecCond} with the oracle bias term in Proposition \ref{PropOracleRisk} and the excess risk:
\begin{lemma}\label{EqPEERBias} We have
\begin{equation*}
\langle \hat P_{>d}f, \Sigma \hat P_{>d}f\rangle=\Big\|\Big(\sum_{k>d}\lambda_k^{1/2}P_k+\sum_{j\le d}\lambda_j^{1/2}P_j\hat P_{>d}-\sum_{k>d}\lambda_k^{1/2}P_k\hat P_{\le d}\Big)f\Big\|^2.
\end{equation*}
\end{lemma}
The squared norm of $\sum_{k>d}\lambda_k^{1/2}P_kf$ is equal to the oracle bias term in Proposition \ref{PropOracleRisk}. On the other hand, the remaining part in Lemma \ref{EqPEERBias} is connected to the excess risk. In fact, \cite[Lemma 2.6]{RW17} says that for any $\mu\in\R$, we have
\begin{align}\label{EqERDec}
{\mathcal E}^{PCA}_d={\mathcal E}^{PCA}_{\le d}(\mu)+{\mathcal E}^{PCA}_{> d}(\mu)
\end{align}
with 
\begin{equation*}
{\mathcal E}^{PCA}_{\le d}(\mu)=\sum_{j\le d}(\lambda_j-\mu)\| P_j\hat{P}_{> d}\|^2
,\quad
{\mathcal E}^{PCA}_{> d}(\mu)=\sum_{k>d}(\mu-\lambda_{k})\| P_k\hat{P}_{\le d}\|^2.
\end{equation*} 
If the terms in the sums are non-negative, then the Hilbert-Schmidt norm can be moved outside the sum, leading to a similar structure as in Lemma~\ref{EqPEERBias}.
\begin{proof}[Proof of Lemma \ref{EqPEERBias}]
By \eqref{EqSpecDec} and the identity $\hat P_{>d}=I-\hat P_{\le d}$, we have
\begin{align*}
\Sigma^{1/2} \hat P_{>d}=\sum_{j\le d}\lambda_j^{1/2}P_j\hat P_{>d}-\sum_{k>d}\lambda_k^{1/2}P_k\hat P_{\le d}+\sum_{k>d}\lambda_k^{1/2}P_k.
\end{align*}
Inserting this into $\langle  \hat P_{>d}f, \Sigma\hat P_{>d}f\rangle=\|\Sigma^{1/2} \hat P_{>d}f\|^2$, the claim follows.
\end{proof}
The next result gives a first quantitative version of Lemma \ref{EqPEERBias}:
\begin{proposition} \label{ResultBiasPre}
For every $r\le d$, we have
\begin{align*}
\langle \hat P_{>d}f,  \Sigma\hat P_{>d}f\rangle&\le \lambda_{r+1}\|f\|^2+\Big\|\sum_{j\le r}(\lambda_j-\lambda_{r+1})^{1/2}P_j\hat P_{>r}\Big\|_\infty^2\|f\|^2.
\end{align*}
In particular, for every $r\le d$, we have
\begin{align*}
\langle \hat P_{>d}f,  \Sigma\hat P_{>d}f\rangle&\le \lambda_{r+1}\|f\|^2+{\mathcal E}^{PCA}_{\le r}(\lambda_{r+1})\|f\|^2.
\end{align*}
\end{proposition}
The flexibility in $r$ allows us to deal with the case of small or vanishing spectral gaps, see e.g.  Corollary \ref{CorApproxPolDecay} below. The first term $\lambda_{r+1}\|f\|^2$ can be replaced by $2\lambda_{r+1}\|P_{>r }f\|^2+2\lambda_{r+1}\|\hat P_{\le r}-P_{\le r}\|_\infty^2\|f\|^2$, cf. \eqref{EqBiasDecER} below. Since this leads to stronger subspace distances, such improvements are not pursued here.

\begin{proof}
Applying \eqref{EqSpecDec}, we have 
\begin{equation*}
\langle\hat P_{>d}f, \Sigma\hat P_{>d}f\rangle=\|\Sigma^{1/2} \hat P_{>d}f\|^2=\sum_{j\le d}\lambda_j\|P_j\hat P_{>d}f\|^2+\sum_{k>d}\lambda_k\|P_k\hat P_{>d}f\|^2.
\end{equation*}
For $r\le d$, the first term on the right-hand is bounded as follows: 
\begin{align*}
\sum_{j\le d}\lambda_j\|P_j\hat P_{>d}f\|^2&\le\sum_{j\le r}(\lambda_j-\lambda_{r+1})\|P_j\hat P_{>d}f\|^2+\lambda_{r+1}\|P_{\le d}\hat P_{>d}f\|^2\nonumber\\
&=\Big\|\sum_{j\le r}(\lambda_j-\lambda_{r+1})^{1/2}P_j\hat P_{>d}f\Big\|^2+\lambda_{r+1}\|P_{\le d}\hat P_{>d}f\|^2.
\end{align*}
Similarly, we have
\begin{equation*}
\sum_{k>d}\lambda_k\|P_k\hat P_{>d}f\|^2\le \lambda_{r+1}\|P_{>d}\hat P_{>d}f\|^2.
\end{equation*}
Using also that $\|P_{\le d}\hat P_{>d}f\|^2+\|P_{>d}\hat P_{>d}f\|^2=\|\hat P_{>d}f\|^2$, we arrive at
\begin{equation}\label{EqBiasDecER}
\langle\hat P_{>d}f, \Sigma\hat P_{>d}f\rangle\le \Big\|\sum_{j\le r}(\lambda_j-\lambda_{r+1})^{1/2}P_j\hat P_{>d}f\Big\|^2+\lambda_{r+1}\|\hat P_{>d}f\|^2.
\end{equation}
Now, the first claim follows from \eqref{EqBiasDecER}, using that $\|A\hat P_{>d}\|_\infty=\|\hat P_{>d}A\|_\infty\le\|\hat P_{>r}A\|_\infty$ for every bounded, self-adjoint operator $A$ on $\mathcal{H}$. Moreover, the second claim follows from the first one by bounding the operator norm by the Hilbert-Schmidt norm.
\end{proof}

Under an additional source condition, we have the following improvement over Proposition \ref{ResultBiasPre}, cf. Corollary \ref{CorOracleRisk} above:
\begin{proposition} \label{ResultBias}
Suppose that $f=\Sigma^sh$ with $s>0$ and $h\in\mathcal{H}$. Then, for every $r\le d$, we have
\begin{align*}
\langle \hat P_{>d}f,  \Sigma\hat P_{>d}f\rangle&\le C\big(\lambda_{r+1}^{1+2s}+{\mathcal E}^{PCA}_{\le r}(\lambda_{r+1})\big)\|h\|^2
\end{align*}
with a constant $C>0$ depending only on $s$ and $\lambda_1$.
\end{proposition}

\begin{proof}
In view of \eqref{EqBiasDecER}, it suffices to show that 
\begin{equation}\label{EqBiasDecER2}
\lambda_{r+1}\|\hat P_{>d}f\|^2\le \lambda_{r+1}\|\hat P_{>r}f\|^2\le C\big(\lambda_{r+1}^{1+2s}+{\mathcal E}^{PCA}_{\le r}(\lambda_{r+1})\big)\|h\|^2
\end{equation}
with a constant $C>0$ depending only on $s$ and $\lambda_1$. First, we have
\begin{align}\label{EqBiasDecER3}
\|\hat P_{>r}f\|^2= \|\hat P_{>r}(P_{>r}+P_{\le r})f\|^2\le 2\|P_{>r}f\|^2+2\|\hat P_{>r}P_{\le r}f\|^2.
\end{align} 
Using that $f=\Sigma^sh$ we get
\begin{equation}\label{EqBiasDecER4}
\|P_{>r}f\|^2\le \lambda_{r+1}^{2s}\|h\|^2
\end{equation}
as well as
\begin{align}
\|\hat P_{>r}P_{\le r}f\|^2&=\|\sum_{j\le r}\lambda_j^{s}\hat P_{>r}P_{j}h\|^2\nonumber\\
&\le 2\|\sum_{j\le r}(\lambda_j^{s}-\lambda_{r+1}^s)\hat P_{>r}P_{j}h\|^2+2\lambda_{r+1}^{2s}\|\hat P_{>r}P_{\le r}h\|^2\nonumber\\
&\le 2\sum_{j\le r}(\lambda_j^{s}-\lambda_{r+1}^s)^2\|P_{j}\hat P_{>r}\|^2\|h\|^2+2\lambda_{r+1}^{2s}\|h\|^2\label{EqBiasDecER5}.
\end{align}
Combining \eqref{EqBiasDecER2}-\eqref{EqBiasDecER5}, the claim follows if we can show that $\lambda_{r+1}(\lambda_j^{s}-\lambda_{r+1}^s)^2\le C(\lambda_j-\lambda_{r+1})$ for every $j\le r$ with a constant $C>0$ depending only on $s$ and $\lambda_1$. For $s\le 1/2$ this follows from $\lambda_{r+1}(\lambda_j^{s}-\lambda_{r+1}^s)^2\le \lambda_{r+1}(\lambda_j^{2s}-\lambda_{r+1}^{2s})\le \lambda_{r+1}^{2s}(\lambda_j-\lambda_{r+1})$, for $s>1/2$ from $(\lambda_j^{s}-\lambda_{r+1}^s)^2\le \lambda_j^{2s}-\lambda_{r+1}^{2s}\le 2s\lambda_j^{2s-1}(\lambda_j-\lambda_{r+1})\le 2s\lambda_1^{2s-1}(\lambda_j-\lambda_{r+1})$.
\end{proof}

\subsection{Analysis of the variance}
The variance term in Lemma \ref{LemBVDecCond} contains trace scalar products with $\Sigma$, similarly as the excess risk in the representation ${\mathcal E}^{PCA}_d=\operatorname{tr}( \Sigma(P_{\le d}-\hat{P}_{\le d}))$, see e.g. \cite[Equation (2.3)]{RW17}. The main difference is the weighting by the $\hat\lambda_j^{-1}$. Grouping together eigenvalues which are of comparable magnitude, one can bound the variance by a constant times $\sigma^2dn^{-1}$ plus remainder terms having a similar structure as the representation of the excess risk given in \eqref{EqERDec}.

\begin{proposition}\label{ResultVar}
Let $0=r_0< r_1< \dots< r_{d'}$ be natural numbers such that $d\le r_{d'}\le C_1d$ and $\lambda_{r_{l}}^{-1}\lambda_{r_{l-1}+1}\le C_2$ for every $l\le d'$. If $\hat \lambda_d\ge \lambda_d/2$, then 
\begin{align*}
&\sum_{j\le d}\hat\lambda_j^{-1}\operatorname{tr} (\hat P_j\Sigma)\\
&\le 2C_1C_2d+2\sum_{2 \le l\le d'}\lambda_{r_{l}}^{-1}\sum_{k\le r_{l-1}}(\lambda_k-\lambda_{r_{l-1}+1})\|\hat P_{J_l}P_k\|_2^2+2\lambda_1\lambda_d^{-1}\sum_{j\le d}\mathbbm{1}(\hat\lambda_j< \lambda_j/2)
\end{align*}
with $J_l=\{r_{l-1}+1,\dots, r_l\}$ and $\hat P_{J_l}=\sum_{j\in J_l}\hat P_j$.
\end{proposition}
\begin{proof}
If $\hat \lambda_d\ge \lambda_d/2$, then 
\begin{align*}
&\sum_{j\le d}\hat\lambda_j^{-1}\operatorname{tr} (\hat P_j\Sigma)\nonumber\\
&=\sum_{j\le d}\hat\lambda_j^{-1}\operatorname{tr} (\hat P_j\Sigma)\mathbbm{1}(\hat\lambda_j\ge \lambda_j/2)+\sum_{j\le d}\hat\lambda_j^{-1}\operatorname{tr} (\hat P_j\Sigma)\mathbbm{1}(\hat\lambda_j< \lambda_j/2)\nonumber\\
& \le 2\sum_{j\le d}\lambda_j^{-1}\operatorname{tr} (\hat P_j\Sigma)+2\lambda_1\lambda_d^{-1}\sum_{j\le d}\mathbbm{1}(\hat\lambda_j< \lambda_j/2),
\end{align*}
where we applied $\hat\lambda_j^{-1}\le \hat\lambda_d^{-1}\le 2\lambda_d^{-1}$ and  $\operatorname{tr} (\hat P_j\Sigma)\le \lambda_1$ in the last inequality. It remains to bound the first term on the right-hand side. We have
\begin{equation}\label{EqVarEq}
\sum_{j\le d}\lambda_j^{-1}\operatorname{tr} (\hat P_j\Sigma)\le \sum_{l\le d'}\lambda_{r_{l}}^{-1}\operatorname{tr} (\hat P_{J_l}\Sigma).
\end{equation}
By \eqref{EqSpecDec} and the fact that orthogonal projectors are idempotent and self-adjoint, we have 
\[
\operatorname{tr} (\hat P_{J_l}\Sigma)=\sum_{k\ge 1}\lambda_k\operatorname{tr} (\hat P_{J_l}P_k)=\sum_{k\ge 1}\lambda_k\|\hat P_{J_l}P_k\|_2^2.
\]
Hence, $\operatorname{tr} (\hat P_{J_1}\Sigma)\le \lambda_1\|\hat P_{J_1}\|_2^2=\lambda_1r_1$ and
\begin{align*}
\operatorname{tr} (\hat P_{J_l}\Sigma)&\le\lambda_{r_{l-1}+1}\|\hat P_{J_l}\|_2^2+\sum_{k\le r_{l-1}}(\lambda_k-\lambda_{r_{l-1}+1})\|\hat P_{J_l}P_k\|_2^2\\
&\le  \lambda_{r_{l-1}+1}(r_l-r_{l-1})+\sum_{k\le r_{l-1}}(\lambda_k-\lambda_{r_{l-1}+1})\|\hat P_{J_l}P_k\|_2^2
\end{align*}
for every $2\le l\le d'$. Inserting this into \eqref{EqVarEq} and using that $r_{d'}\le C_1d$ and $\lambda_{r_{l}}^{-1}\lambda_{r_{l-1}+1}\le C_2$ for every $l\le d'$, the claim follows.
\end{proof}

\subsection{Bounds for the prediction error}
Combining Lemma \ref{LemBVDecCond} with Propositions \ref{ResultBias} and \ref{ResultVar}, we obtain an upper bound for the mean squared prediction error of PCR conditional on the design. This bound has the same leading terms appearing in Corollary \ref{CorOracleRisk} for the oracle PCR estimator. In Section \ref{SecRW17}, we will see how the remainder terms can be bounded using the excess risk bounds from \cite{RW17}. 
\begin{thm} \label{FinalResult} Suppose that $f=\Sigma^sh$ with $s\ge 0$ and $h\in\mathcal{H}$. Grant the assumptions of Proposition \ref{ResultVar}. Then, for each $r\le d$, we have on the event $\{\hat \lambda_d\ge \lambda_d/2\}$,
\begin{align*}
\mathbb{E}\big(\|\EstPCR-f\|_{L^2(\mathbb{P}^X)}^2|X_1,\dots,X_n\big)&\le C\Big(\lambda_{r+1}^{1+2s}\|h\|^2+R_1+\frac{\sigma^2}{n}(d+R_2)\Big)
\end{align*}
with a constant $C>0$ depending only on $s$, $\lambda_1$ and $C_1$, $C_2$ from Proposition~\ref{ResultVar}, and remainder terms $R_1={\mathcal E}^{PCA}_{\le r}(\lambda_{r+1})\|h\|^2$ and 
\[
R_2=\sum_{2\le l\le d'}\lambda_{r_{l}}^{-1}\sum_{k\le r_{l-1}}(\lambda_k-\lambda_{r_{l-1}+1})\|\hat P_{J_l}P_k\|_2^2+\lambda_d^{-1} \sum_{j\le d}\mathbbm{1}(\hat\lambda_j< \lambda_j/2).
\]
\end{thm}
While the prediction error is standard in statistical learning, estimates in $\mathcal{H}$-norm are standard for inverse problems. Following a similar but simpler line of arguments, we get the following upper bound in $\mathcal{H}$ norm. We observe that the excess risk has to be replaced by stronger subspace distances.
\begin{proposition} Suppose that $f=\Sigma^sh$ with $s\ge 0$ and $h\in\mathcal{H}$. Then, for every $r\le d$, we have on the event $\{\hat \lambda_d\ge \lambda_d/2\}$, 
\begin{align*}
\mathbb{E}\big(\|\EstPCR-f\|^2|X_1,\dots,X_n\big)&\le 2\lambda_{r+1}^{2s}\| h\|^2+\frac{2\sigma^2}{n}\Big(\sum_{j\le d}\lambda_j^{-1}+R_{1}\Big)+R_2
\end{align*}
with 
\[
R_{1}=\lambda_d^{-1} \sum_{j\le d}\mathbbm{1}(\hat\lambda_j< \lambda_j/2),\quad R_2=2\|\hat P_{\le r}- P_{\le r}\|_\infty^2\|f\|^2.
\]
\end{proposition}

\section{Bounds for the mean squared prediction error}
\subsection{Bounds for the excess risk}\label{SecRW17}
In this section, we state the main perturbation bounds for the excess risk obtained in \cite{RW17}. From now on, we suppose that $X$ is sub-Gaussian. In order to deal with weaker moment assumptions one can alternatively use the results obtained in \cite{JW18b}, though under stronger eigenvalue assumptions.
\begin{assumption}\label{SubGauss}
Suppose that there is a constant $\SGN$ such that
\[
\forall g\in\mathcal{H},\quad \sup_{k\ge 1}k^{-1/2}\mathbb{E}\big[ |\langle g,X\rangle|^{k} \big] ^{1/k} \le \SGN\mathbb{E}[\langle g,X\rangle^2]^{1/2}.
\]
\end{assumption}
First, \cite[Corollary 2.8]{RW17} implies that 
\[
\mathbb{E}{\mathcal E}^{PCA}_{\le r}(\lambda_{r+1})\le C\sum_{j\le r}\frac{\lambda_j\tr (\Sigma)}{n(\lambda_j-\lambda_{r+1})}
\]
with a constant $C$ depending only on $\SGN$. If for some $\alpha>0$ 
\begin{equation}\label{EqExpDec}
\lambda_j=e^{-\alpha j},\quad j\ge 1,
\end{equation} 
then $ \lambda_j-\lambda_{r+1}\ge (1-e^{-\alpha})\lambda_j$ for every $j\le r$, implying that
$\mathbb{E}{\mathcal E}^{PCA}_{\le r}(\lambda_{r+1})\le Crn^{-1}$ with a constant $C>0$ depending only on $\SGN$ and $\alpha$.
Similarly, if for some $\alpha>1$ 
\begin{equation}\label{EqPolDec}
\lambda_j=j^{-\alpha},\quad j\ge 1,
\end{equation}
then we have $j\lambda_j\ge (r+1)\lambda_{r+1}$ and thus $\lambda_j/(\lambda_j-\lambda_{r+1})\le (r+1)/(r+1-j)$ for every $j\le r$, implying that $\mathbb{E}{\mathcal E}^{PCA}_{\le r}(\lambda_{r+1}) \le Crn^{-1}\log (er)$. While the latter bound is only up to a logarithmic term of larger order than the variance term $\sigma^2dn^{-1}$ (for $r=d$), \cite[Proposition 2.10]{RW17} states the following improvement, valid under additional eigenvalue conditions:
\begin{thm}\label{ThmERBound}  Grant Assumption \ref{SubGauss}. Suppose that 
\begin{equation}\label{ThmERBoundCond}
\frac{\lambda_{r}}{\lambda_{r}-\lambda_{r+1}} \sum_{j\le r}\frac{\lambda_j}{\lambda_j-\lambda_{r+1}} \le cn.
\end{equation}
Then we have
\begin{equation*}
\mathbb{E}{\mathcal E}^{PCA}_{\le r}(\lambda_{r+1})\le C\sum_{j\le r}\frac{\lambda_j\operatorname{tr}_{>r}(\Sigma)}{n(\lambda_j-\lambda_{r+1})}+C\sum_{j\le r}\frac{\lambda_j\operatorname{tr}(\Sigma)}{n(\lambda_j-\lambda_{r+1})}e^{-cn(\lambda_{r}-\lambda_{r+1})^2/\lambda_{r}^2}
\end{equation*}
with $\operatorname{tr}_{>r}(\Sigma)=\sum_{k>r}\lambda_k$ and constants $c,C>0$ depending only on $\SGN$.
\end{thm}
If \eqref{EqExpDec} holds, then Theorem \ref{ThmERBound} with $r=d$ gives
\begin{equation}\label{EqERExp}
\forall d\le cn,\quad\mathbb{E}{\mathcal E}^{PCA}_{\le d}(\lambda_{d+1})\le C\frac{d e^{-\alpha d}}{n}.
\end{equation}
Similarly, if \eqref{EqPolDec} holds, then Theorem \ref{ThmERBound} with $r=d$ gives
\begin{align}
\mathbb{E}{\mathcal E}^{PCA}_{\le d}(\lambda_{d+1})&
\le C\frac {d\log (ed)}{n}\big(\operatorname{tr}_{>d}(\Sigma)+\operatorname{tr}(\Sigma)e^{-cn/d^2}\big)\nonumber\\
&\le C\frac {d^{2-\alpha}\log (ed)}{n},\label{EqERPol}
\end{align}
provided that $d^2\log (ed)\le cn$. In both inequalities, $c,C>0$ are constants depending only on $\SGN$ and $\alpha$. 

The first remainder term in Proposition~\ref{ResultVar} has a similar structure as the representation of the excess risk given in \eqref{EqERDec}. Using the techniques in \cite{RW17} leading to Theorem \ref{ThmERBound}, we prove in Appendix \ref{ThmVarBoundProof} the following upper bound: 
\begin{thm}\label{ThmVarBound}
For natural numbers $s>r>0$, let $J=\{r+1,\dots,s\}$.  Suppose that 
\begin{equation}\label{ThmVarBoundCond}
\Big(\max_{k\ge 1}\frac{\lambda_k}{g_k}\Big) \Big(\sum_{k\ge 1}\frac{\lambda_k}{g_k}\Big)\le cn
\end{equation}
with $g_k=\lambda_k-\lambda_{r+1}$ for $k\le r$, $g_k=\min(\lambda_r-\lambda_k,\lambda_k-\lambda_{s+1})$ for $k\in J$, and $g_k=\lambda_{s}-\lambda_{k}$ for $k>s$. Then we have
\begin{equation*}
\mathbb{E}\sum_{k\notin J}g_k\|\hat P_{J}P_k\|_2^2\le C\sum_{k\notin J}\frac{\lambda_k\operatorname{tr}_J(\Sigma)}{ng_k}+R
\end{equation*}
with $\operatorname{tr}_J(\Sigma)=\sum_{j\in J}\lambda_j$ and $R=C\lambda_1(s-r)^2\exp(-cn\min_{k\ge 1}g_k^2/\lambda_k^2)$. 

In particular, if $\lambda_s^{-1}\lambda_{r+1}\leq C_1$, then we have
\begin{equation}\label{ThmExpVarBound}
\mathbb{E}\lambda_{s}^{-1}\sum_{k\notin J}g_k\|\hat P_{J}P_k\|_2^2\le C_1C(s-r)+\lambda_{s}^{-1}R.
\end{equation}
Here, $c,C>0$ are constants depending only on $\SGN$.
\end{thm}
Turning to the setting of Proposition \ref{ResultVar}, suppose that \eqref{ThmVarBoundCond} holds for each pair $(r,s)=(r_{l-1},r_{l})$, $2\le l\le d'$. Then \eqref{ThmExpVarBound} yields that the expectation of the first remainder term in Proposition \ref{ResultVar} is bounded by a constant times $d$ plus an (exponentially) small remainder term. A simple choice is given by $r_j=j$, $j\le d$. In this case, \eqref{ThmVarBoundCond} with $J=\{j\}$ is implied by 
\[
\frac{\lambda_j}{g_j}\Big(\frac{\lambda_j}{g_j}+\sum_{k\neq j}\frac{\lambda_k}{|\lambda_j-\lambda_k|}\Big)\le cn,\qquad g_j=\min(\lambda_{j-1}-\lambda_{j},\lambda_j-\lambda_{j+1}).
\]
If e.g. \eqref{EqExpDec} holds, then Theorem \ref{ThmVarBound} gives
\begin{align}
\quad&\mathbb{E}\sum_{j\le d}\lambda_j^{-1}\sum_{k\leq j-1}(\lambda_k-\lambda_j)\|\hat P_jP_k\|_2^2\leq \mathbb{E}\sum_{j\le d}\lambda_j^{-1}\sum_{k\neq j}|\lambda_j-\lambda_k|\|\hat P_jP_k\|_2^2\le Cd.\label{EqVarRemExp}
\end{align}
for all $d\le cn$. Similarly, if \eqref{EqPolDec} holds, then Theorem \ref{ThmVarBound} gives
\begin{equation}\label{EqVarRemPol}
\mathbb{E}\sum_{j\le d}\lambda_j^{-1}\sum_{k\leq j-1}(\lambda_k-\lambda_j)\|\hat P_jP_k\|_2^2\le Cd+Cd\lambda_1\lambda_d^{-1}e^{-cn/d^2}\le Cd,
\end{equation}
provided that $d^2\log (ed)\le cn$. In both inequalities, $c,C>0$ are constants depending only on $\SGN$ and $\alpha$. 
We conclude this section by stating \cite[Theorem 2.15]{RW17} (applied with $y=1/2$), needed to deal with the second remainder term in Proposition \ref{ResultVar}.
\begin{proposition}\label{PropEVConc} Grant Assumption \ref{SubGauss}. Then there are constants $c_1,c_2>0$ depending only on $\SGN$ such that 
\begin{equation*}
\forall j\le c_1n, \quad\mathbb{P}(\hat\lambda_j< \lambda_j/2)\le e^{1-c_2n}.
\end{equation*}
\end{proposition}

\subsection{Applications}\label{SecSMA}
Let us apply our results to standard model assumptions on the $\lambda_j$ and $f$. 
\subsubsection{The isotropic case}
Suppose that $\mathcal{H}=\mathbb{R}^p$ and that $\Sigma=I_p$. Then Propositions \ref{ResultBiasPre} and \ref{ResultVar}, in combination with Proposition \ref{PropEVConc} and the identity ${\mathcal E}^{PCA}_{\le d}(\lambda_{d+1})=0$, lead to the (trivial) bound 
\[
\forall d\le c_1n,\quad \mathbb{E}\|\EstTPCR-f\|_{L^2(\mathbb{P}^X)}^2\le \Big(\|f\|^2+\frac{2\sigma^2d}{n}\Big)(1+e^{1-c_2n}).
\]
Note that the bias term cannot be further improved for $d$ not too large. In fact, suppose that $X$ is Gaussian. Then the empirical eigenbasis (considered as an orthogonal matrix, the sign of each column chosen uniformly at random) is distributed according to the Haar measure on the orthogonal group (see e.g. \cite[Theorem 5.3.1]{F85}). In particular each eigenvector is distributed according to the uniform measure $\mu$ on the $(p-1)$-sphere $S^{p-1}$ and we get
\[
\mathbb{E}\langle  \hat P_{>d}f, \Sigma\hat P_{>d}f\rangle=\mathbb{E}\langle\hat P_{>d}f,f\rangle=(p-d)\int_{S^{p-1}}\langle u,f\rangle^2\,d\mu(u)=\frac{p-d}{p}\|f\|^2.
\]

\subsubsection{Exponential decay}
Suppose that \eqref{EqExpDec} holds. Then, applying Theorem \ref{FinalResult}, \eqref{EqERExp}, \eqref{EqVarRemExp}, and Proposition \ref{PropEVConc}, we get
\[
\forall d\le cn,\quad \mathbb{E}\|\EstTPCR-f\|_{L^2(\mathbb{P}^X)}^2\le C\Big(e^{-\alpha d}\|f\|^2+\frac{\sigma^2d}{n}\Big),
\]
where $c,C>0$ are constants depending only on $\SGN$ and $\alpha$. In particular, choosing $d$ of size $\log (n)/\alpha $ gives

\begin{equation*}
\sup_{f\in\mathcal{H}:\|f\|\le 1}\mathbb{E}_f\|\EstTPCR-f\|_{L^2(\mathbb{P}^X)}^2\le C\frac{\log (n)}{n}
\end{equation*}
with a constant $C>0$ depending only on $\SGN$, $\alpha$, and $\sigma^2$. This rate is minimax optimal, see e.g. \cite{CJ10}.

\subsubsection{Polynomial decay}
Suppose that \eqref{EqPolDec} holds and that $f$ satisfies the source condition
\begin{equation*}
f\in H(s,L)=\{g\in\mathcal{H}:g=\Sigma^sh,\|h\|\le L\},\quad s \ge 0,L>0,
\end{equation*}
which is equivalent to assuming that $f$ lies in a Sobolev-type ellipsoid. Then, applying Theorem \ref{FinalResult}, \eqref{EqERPol}, \eqref{EqVarRemPol}, and Proposition \ref{PropEVConc}, we get 
\[
\mathbb{E}\|\EstTPCR-f\|_{L^2(\mathbb{P}^X)}^2\le C\Big(d^{-\alpha-2s\alpha}L^2+\frac{d^{2-\alpha}\log(ed)}{n}L^2+\frac{\sigma^2d}{n}\Big),
\]
provided that $d^2\log(ed)\le cn$, where $c,C>0$ are constants depending only on $\SGN$ and $\alpha$, and $s$. In particular, choosing $d$ of size $n^{1/(2s\alpha+\alpha+1)}$ leads to
\begin{equation*}
\sup_{f\in H(r,L)}\mathbb{E}_f\|\EstTPCR-f\|_{L^2(\mathbb{P}^X)}^2\le Cn^{-(2s\alpha+\alpha)/(2s\alpha+\alpha+1)}
\end{equation*}
with a constant $C>0$ depending only on $\SGN$, $\alpha$, $s$, $L$, and $\sigma^2$. This rate is minimax optimal, see \cite{HH07,CJ10} or \cite{BM18}.

\subsubsection{Convex structures}
\begin{assumption}\label{AssEVConv} Suppose that there is a constant $C_1>0$ such that 
\begin{equation}\label{EqEVConvCond}
\forall j\le d+1,\quad\sum_{k\neq j}\frac{\lambda_k}{|\lambda_j-\lambda_k|}\le C_1j\log(ej),\quad \frac{\lambda_{j-1}}{\lambda_{j-1}-{\lambda_{j}}}\le j.
\end{equation}
\end{assumption}
Assumption \ref{AssEVConv} is satisfied under \eqref{EqPolDec} and \eqref{EqExpDec}. It holds under quite general conditions, see e.g. \cite[Lemma 1]{CMS07} and \cite[Lemma 10.1]{HMV13}. 
\begin{corollary}\label{CorConvEV} Grant Assumptions \ref{SubGauss} and \ref{AssEVConv}. Suppose that $f=\Sigma^sh$ with $s\ge 0$ and $h\in\mathcal{H}$. Then there are constants $c,C>0$ depending only on $\lambda_1$, $s$, $\SGN$, and $C_1$ such that the following holds true. For all $d\ge 1$ such that $d^2\log (ed)\le cn$, we have
\[
\mathbb{E}\|\EstTPCR-f\|_{L^2(\mathbb{P}^X)}^2\le C\Big(\lambda_{d+1}^{1+2s}\|h\|^2+\frac{\sigma^2d}{n}+R\Big)
\]
with
\[
R=\frac{d\log (ed)}{n}\operatorname{tr}_{>d}(\Sigma)\|h\|^2+\Big(\operatorname{tr}(\Sigma)\|h\|^2+\lambda_d^{-1}\frac{\sigma^2d}{n}\Big)e^{-cn/d^2}.
\]
\end{corollary}
\begin{proof}
Corollary \ref{CorConvEV} follows from Theorem \ref{FinalResult} (applied with $r_j=j$, $j\le d$) in combination with the results from Section \ref{SecRW17}. Note that the first inequalities in \eqref{EqERPol} and \eqref{EqVarRemPol} also hold under Assumption \ref{AssEVConv}. 
\end{proof}

\subsubsection{Approximate polynomial decay}
\begin{assumption}\label{ApproxPolDecay}
Suppose that there are constants $\alpha>1$ and $\CEV\ge 1$ such that
\[
\forall j\ge 1,\qquad \CEV^{-1}j^{-\alpha}\le \lambda_j\le \CEV j^{-\alpha}.
\]
\end{assumption}
The following corollary shows that the results for polynomial decay still hold true under Assumption \ref{ApproxPolDecay}. 
\begin{corollary}\label{CorApproxPolDecay} Grant Assumptions \ref{SubGauss} and \ref{ApproxPolDecay}. Suppose that $f=\Sigma^sh$ with $s\ge 0$ and $h\in\mathcal{H}$. Then, there are constants $c,C>0$ depending only on $\CEV$, $\alpha$, $s$, and $\SGN$ such that the following holds true. For all $d\ge 1$ such that $d^2\log(ed)\le cn$,  we have
\[
\mathbb{E}\|\EstTPCR-f\|_{L^2(\mathbb{P}^X)}^2\le C\Big(d^{-\alpha-2s\alpha}\|h\|^2+\frac{d^{2-\alpha}\log(ed)}{n}\|h\|^2+\frac{\sigma^2d}{n}\Big).
\]
\end{corollary}
In order to prove Corollary \ref{CorApproxPolDecay} we have to use the flexibility in $r$ and $(r_l)$ in Propositions \ref{ResultBias} and \ref{ResultVar}, respectively. A crucial ingredient is the following lemma proved in Appendix \ref{ProofLemApproxPolDecay1}.
\begin{lemma}\label{LemApproxPolDecay1}
(a) There are constants $c_1,C_1>0$ depending only on $\CEV$ and $\alpha$ such that the following holds true. For every $d\ge 1/c$ there is a natural number $r\in[c_1d, d]$ such that 
\[
\sum_{j\le r}\frac{\lambda_j}{\lambda_j-\lambda_{r+1}}\le C_1r\log(er),\qquad \frac{\lambda_r}{\lambda_r-\lambda_{r+1}}\le C_1r.
\]
(b) There are constants $C_1,C_2>0$ depending only on $\CEV$ and $\alpha$ such that the following holds true. For every $d\ge 1$ there is a natural number $r\in[d, C_1d]$ such that 
\begin{equation}\label{EqApproxPolDecay2}
\sum_{j\le r}\frac{\lambda_j}{\lambda_j-\lambda_{r+1}}+\sum_{k> r}\frac{\lambda_k}{\lambda_r-\lambda_{k}}\le C_2r\log(er),\qquad \frac{\lambda_r}{\lambda_r-\lambda_{r+1}}\le C_2r.
\end{equation}
\end{lemma}
\begin{proof}[Proof of Corollary \ref{CorApproxPolDecay}]
Corollary \ref{CorApproxPolDecay} follows from Theorem \ref{FinalResult} in combination with the results from Section \ref{SecRW17} and Lemma \ref{LemApproxPolDecay1}. Let $c_1,C_1$ be the constants from Lemma \ref{LemApproxPolDecay1}. We may assume that $d\ge 1/c_1$, because the result is clear for $d<1/c_1$. By Lemma \ref{LemApproxPolDecay1}(a) and the assumption on $d$ we can find an $r\in[c_1d,d]$ such that \eqref{ThmERBoundCond} holds and we have
\begin{equation}\label{EqERApproxPolDecay}
\mathbb{E}{\mathcal E}^{PCA}_{\le r}(\lambda_{r+1})\le C\frac{r\log(er)}{n}\big(r^{1-\alpha}+e^{-n/r^2}\big)\le C\frac{d^{2-\alpha}\log(ed)}{n}.
\end{equation}
Hence, for this choice of $r$, we get 
\[
\lambda_{r+1}^{1+2s}\|h\|^2+\mathbb{E}{\mathcal E}^{PCA}_{\le r}(\lambda_{r+1})\|h\|^2\le Cd^{-\alpha-2s\alpha}\|h\|^2+C\frac{d^{2-\alpha}\log(ed)}{n}\|h\|^2
\]
with a constant $C$ depending only on $\CEV$, $\alpha$, $s$, and $\SGN$.
It remains to bound $R_2$ from Theorem \ref{FinalResult}. First, the expectation of the second summand in $R_2$ is exponentially small in $n$ by using Proposition \ref{PropEVConc}. Let us consider the first summand. By Lemma \ref{LemApproxPolDecay1}(b), we can find a sequence $0=r_0<r_1<\dots<r_{d'}$ such that $r_{d'}\ge d$, $r_{l}\in [r_{l-1}+1,C_1(r_{l-1}+1)]$, and \eqref{EqApproxPolDecay2} holds for $r=r_l$, $l=1,\dots,d'$. Using also Assumption \ref{ApproxPolDecay}, we get $\lambda_{r_{l}}^{-1}\lambda_{r_{l-1}+1}\le \CEV^2C_1^\alpha$. Moreover, \eqref{ThmVarBoundCond} is satisfied for each pair $(r,s)=(r_{l-1},r_{l})$, $2\le l\le d'$ and \eqref{ThmExpVarBound} leads to
\begin{align}
&\mathbb{E}\sum_{2 \le l\le d'}\lambda_{r_{l}}^{-1}\sum_{k\leq r_{l-1}}(\lambda_k-\lambda_{r_{l-1}+1})|\|\hat P_{J_l}P_k\|_2^2\nonumber\\
&\leq \mathbb{E}\sum_{2 \le l\le d'}\lambda_{r_{l}}^{-1}\sum_{k\notin J_l}g_k\|\hat P_{J_l}P_k\|_2^2\le Cd+d^{2+\alpha}e^{-cn/d^2}\le Cd\label{EqVarApproxPolDecay}
\end{align}
with a constant $C$ depending only on $\CEV$, $\alpha$, and $\SGN$.
The claim now follows from inserting \eqref{EqERApproxPolDecay} and \eqref{EqVarApproxPolDecay} into Theorem \ref{FinalResult}, in combination with Proposition \ref{PropEVConc} and the definition of $\EstTPCR$.
\end{proof}

\section*{Acknowledgements}
The research has been partially funded by Deutsche Forschungsgemeinschaft (DFG) through grant CRC 1294 ``Data Assimilation'', Project (A4) ``Nonlinear statistical inverse problems with random observations''.

\bibliographystyle{plain}
\bibliography{lit}

\appendix
\section{Additional proofs}
\subsection{Proof of Proposition \ref{PropOracleRisk}}\label{PropOracleRiskProof}
We abbreviate $\EstTORA$ and $\EstORA$ by $\tilde f_d$ and $\hat f_d$, respectively. Consider $X'=P_{\le d}X$, $X_i'=P_{\le d}X_i$ (defined on $U_d$) which lead to covariance and sample covariance $\Sigma'=P_{\le d}\Sigma P_{\le d}$, $\hat\Sigma'=P_{\le d}\hat\Sigma P_{\le d}$. In the following, we use the notation introduced in Sections \ref{SecPCA}, \ref{SecPCR} with an additional superscript $'$. With this notation, we have 
\begin{equation}\label{EqDecOracleEst}
\hat f_d=n^{-1/2}\sum_{j\le d}\hat\lambda_j'^{-1/2}\langle \mathbf{Y},\hat{v}_j'\rangle \hat u_j',
\end{equation}
provided that $\hat\lambda_d'>0$. 
We define the events $\mathcal{A}_d=\{\hat\lambda_d'\ge \lambda_d/2\}$ and
\[
\mathcal{E}_d=\{(1/2)\|g\|_{L^2(\mathbb{P}^X)}^2\le \|g\|_n^2\le (3/2)\|g\|_{L^2(\mathbb{P}^X)}^2\,\forall g\in U_d\}.
\] 
Inserting $\|g\|_n^2=\langle g,\hat\Sigma' g\rangle $, $\|g\|_{L^2(\mathbb{P}^X)}^2=\langle g,\Sigma' g\rangle$, $g\in U_d$, and applying \cite[Theorem~1]{KL14}, there are constants $c_1,c_2>0$ depending only on $\SGN$ such that 
\[
\mathbb{P}(\mathcal{E}_d^c)=\mathbb{P}(\|\Sigma'^{-1/2}(\hat\Sigma'-\Sigma')\Sigma'^{-1/2}\|_\infty >1/2)\le e^{-c_1n},
\]
provided that $d\le c_2n$. We now decompose the mean squared prediction error as follows:
\begin{align*}
\mathbb{E}\|\tilde f_d-f\|_{L^2(\mathbb{P}^X)}^2&\le \mathbb{E}\mathbbm{1}_{\mathcal{A}_d\cap \mathcal{E}_d}\|\hat f_d-P_{\le d}f\|_{L^2(\mathbb{P}^X)}^2+\mathbb{E}\mathbbm{1}_{\mathcal{A}_d\cap \mathcal{E}_d^c}\|\hat f_d-P_{\le d}f\|_{L^2(\mathbb{P}^X)}^2\\
&+\mathbb{E}\mathbbm{1}_{\mathcal{A}_d^c}\|P_{\le d}f\|_{L^2(\mathbb{P}^X)}^2+\|P_{> d}f\|_{L^2(\mathbb{P}^X)}^2=:I_1+\dots +I_4.
\end{align*}
The term $I_4$ is exactly the oracle bias term. By Proposition \ref{PropEVConc}, we have
\[
I_3\le \|f\|_{L^2(\mathbb{P}^X)}^2e^{1-c_1n},
\]
provided that $d\le c_2n$. Moreover, 
\begin{align*}
I_2&\le 2\mathbb{P}(\mathcal{E}_d^c)\|f\|_{L^2(\mathbb{P}^X)}^2+2\mathbb{E}\mathbbm{1}_{\mathcal{A}_d\cap \mathcal{E}_d^c}\|\hat f_d\|_{L^2(\mathbb{P}^X)}^2.
\end{align*}
Inserting that on $\mathcal{A}_d$,
\[
\|\hat f_d\|_{L^2(\mathbb{P}^X)}^2\le \lambda_{1}\|\hat f_d\|^2\le 2\lambda_{1}\lambda_{d}^{-1}\|\mathbf{Y}\|_n^2\le 4\lambda_{1}\lambda_{d}^{-1}(\|f\|_n^2+\|\boldsymbol{\epsilon}\|_n^2),
\]
we get
\begin{align*}
I_2\le 2\mathbb{P}(\mathcal{E}_d^c)\|f\|_{L^2(\mathbb{P}^X)}^2+8\lambda_{1}\lambda_{d}^{-1}\mathbb{P}^{1/2}(\mathcal{E}_d^c)\mathbb{E}^{1/2}\|f\|_n^4+8\lambda_{1}\lambda_{d}^{-1}\mathbb{P}(\mathcal{E}_d^c)\sigma^2.
\end{align*}
By the Cauchy-Schwarz inequality and Assumption \ref{SubGauss}, we get
\[
\mathbb{E}^{1/2}\|f\|_n^4\le \mathbb{E}^{1/2}\langle f,X\rangle^4\le 4\SGN^2\|f\|_{L^2(\mathbb{P}^X)}^2
\]
and thus
\[
I_2\le (32\SGN^2+10)\lambda_{1}\lambda_{d}^{-1}(\sigma^2+\|f\|_{L^2(\mathbb{P}^X)}^2)e^{-c_1n}.
\]
Finally, 
\begin{align*}
(1/2)I_1\le \mathbb{E}\mathbbm{1}_{\mathcal{A}_d\cap \mathcal{E}_d}\|\hat f_d-P_{\le d}f\|_n^2.
\end{align*}
Letting $\hat\Pi_n:\mathbb{R}^n\rightarrow S_nU_d,\mathbf{y}\mapsto \argmin_{g\in  U_d}\|\mathbf{y}-S_ng\|_n^2$ be the orthogonal projection onto $S_nU_d$, we have 
\begin{align*}
(1/2)I_1&\le \mathbb{E}\|\hat\Pi_n f-P_{\le d}f\|_n^2+\mathbb{E}\mathbb{E}(\|\hat\Pi_n\boldsymbol{\epsilon}\|_n^2|X_1,\dots,X_n)\\
&\le \mathbb{E}\| f-P_{\le d}f\|_n^2+\sigma^2dn^{-1}=\| P_{> d}f\|_{L^2(\mathbb{P}^X)}^2+\sigma^2dn^{-1},
\end{align*}
where we applied the projection theorem in the second inequality. The claim follows from collecting all these inequalities. \qed
\subsection{Proof of Theorem \ref{ThmVarBound}}\label{ThmVarBoundProof}
The following lemma is an extension of \cite[Proposition 3.5]{RW17}:
\begin{lemma}\label{lem1} Let $\Delta=\Sigma-\hat\Sigma$, $J=\{r+1,\dots,s\}$, $g_k=\min_{j\in J}|\lambda_j-\lambda_k|$ for $k\notin J$, $\hat P_{J}=\sum_{j\in J}\hat P_j$, and 
\[
S_{J^c}=\sum_{k\notin J}g_k^{-1/2}P_k.
\]
Then, on the event  $\bigcap_{j\in J,k\notin J}\{|\hat\lambda_j-\lambda_k|\ge |\lambda_j-\lambda_k|/2\}\cap \{\|S_{J^c}\Delta S_{J^c}\|_\infty\le 1/4\}$, we have
\[
\sum_{k\notin J}g_k\|\hat P_{J}P_k\|_2^2\le 16\sum_{k\notin J}\frac{\|P_J\Delta P_k\|_2^2}{g_k}.
\]
\end{lemma}
\begin{proof}
On the event $\bigcap_{j\in J,k\notin J}\{|\hat\lambda_j-\lambda_k|\ge |\lambda_j-\lambda_k|/2\}$ it follows from \cite[Equation (3.1)]{RW17} and the definition of $g_k$ that
\begin{equation}\label{EqLem6Gen}
\forall j\in J,k\notin J,\quad \|\hat P_jP_k\|_2^2\le \frac{4\|\hat P_j \Delta P_k\|_2^2}{(\lambda_j-\lambda_k)^2}\leq  \frac{4\|\hat P_j \Delta P_k\|_2^2}{g_k^2}
\end{equation}
and thus
\begin{equation}\label{EqErrorDec}
\sum_{k\notin J}g_k\|\hat P_JP_k\|_2^2\le 4\sum_{k\notin J}\frac{\|\hat P_J \Delta P_k\|_2^2}{g_k}.
\end{equation}
We introduce the operators
\[
P_{J^c}=\sum_{k\notin J}P_k\quad\text{and}\quad T_{J^c}=\sum_{k\notin J}g_k^{1/2}P_k,
\]
which satisfy the identities $T_{J^c}S_{J^c}=P_{J^c}$ and
\[
\|\hat P_J T_{J^c}\|_2^2=\sum_{k\notin J}g_k\|\hat P_JP_k\|_2^2.
\]
Hence, \eqref{EqErrorDec} can be written as
\begin{equation}\label{EqRec1}
\|\hat P_JT_{J^c}\|_2^2\le 4\|\hat P_J \Delta S_{J^c}\|_2^2.
\end{equation}
On the event $\{\|S_{J^c}\Delta S_{J^c}\|_\infty\le 1/4\}$, the last term is bounded via
\begin{align}
\|\hat P_J ES_{J^c}\|_2^2
&\le 2\|\hat P_J P_J \Delta S_{J^c}\|_2^2+2\|\hat P_J P_{J^c} \Delta S_{J^c}\|_2^2\nonumber\\
&=2\|\hat P_J P_J \Delta S_{J^c}\|_2^2+2\|\hat P_J T_{J^c}S_{J^c} \Delta S_{J^c}\|_2^2\nonumber\\
&\le 2\| P_J \Delta S_{J^c}\|_2^2+2\|\hat P_J T_{J^c}\|_2^2\|S_{J^c} \Delta S_{J^c}\|_\infty^2\nonumber\\
&\le 2\| P_J \Delta S_{J^c}\|_2^2+\|\hat P_J T_{J^c}\|_2^2/8\label{EqRec2}.
\end{align}
Inserting \eqref{EqRec2} into \eqref{EqRec1}, we get
\[
\|\hat P_JT_{J^c}\|_2^2\le 8\| P_J ES_{J^c}\|_2^2+\|\hat P_JT_{J^c}\|_2^2/2,
\]
and the claim follows.
\end{proof}
\begin{lemma}\label{Lem2} Grant Assumption \ref{SubGauss}. Let $\mathcal{E}=\bigcap_{j\in J,k\notin J}\{|\hat\lambda_j-\lambda_k|\ge |\lambda_j-\lambda_k|/2\}\cap \{\|S_{J^c}\Delta S_{J^c}\|_\infty\le 1/4\}$. If \eqref{ThmVarBoundCond} holds, then we have 
\[
\mathbb{P}(\mathcal{E}^c)\le (2(s-r)+1)\exp\Big(-cn\min\Big(\frac{(\lambda_r-\lambda_{r+1})^2}{\lambda_r^2},\frac{(\lambda_s-\lambda_{s+1})^2}{\lambda_s^2}\Big)\Big)
\]
with a constant $c>0$ depending only on $\SGN$.
\end{lemma}
\begin{proof}
Proceeding as in \cite[Lemma 9]{RW17}, we have
\[
\mathbb{P}(\|S_{J^c}\Delta S_{J^c}\|_\infty>1/4)\le \exp\Big(-n\min\Big(\frac{(\lambda_r-\lambda_{r+1})^2}{16C_3^2\lambda_r^2},\frac{(\lambda_s-\lambda_{s+1})^2}{16C_3^2\lambda_{s+1}^2}\Big)\Big),
\]
provided that 
\begin{equation*}
\max\bigg(\frac{\lambda_r}{\lambda_r-\lambda_{r+1}},\frac{\lambda_{s+1}}{\lambda_s-\lambda_{s+1}}\bigg)\bigg(\sum_{j\le r}\frac{\lambda_j}{\lambda_{j}-\lambda_{r+1}}+\sum_{k>s}\frac{\lambda_k}{\lambda_{s}-\lambda_k}\bigg)\le \frac{n}{16C_3^2},
\end{equation*}
where $C_3$ is the constant from \cite[Equation (3.16)]{RW17} depending only on $C_1$. The latter is satisfied under \eqref{ThmVarBoundCond} if $c\le 1/(16C_3^2)$. We turn to the empirical eigenvalues. First, the condition that $|\hat\lambda_j-\lambda_k|\ge |\lambda_j-\lambda_k|/2$ for every $j\in J,k\notin J$ is implied by the condition  that $\hat\lambda_j-\lambda_k\ge (\lambda_j-\lambda_k)/2$ for every $j\in J,k>s$ and $\lambda_k-\hat\lambda_j\ge (\lambda_k-\lambda_j)/2$ for every $j\in J,k\le r$. It is simple to check that the latter condition is equivalent to $-(\lambda_j-\lambda_{s+1})/2\le \hat\lambda_j-\lambda_j\le (\lambda_r-\lambda_j)/2$ for every $j\in J$. By \cite[Corollary 3.14]{RW17}, we have for $j\in J$,
\begin{align*}
&\mathbb{P}(\hat\lambda_j-\lambda_j<-(\lambda_j-\lambda_{s+1})/2)\\
&\le \exp\Big(-\frac{n(\lambda_j-\lambda_{s+1})^2}{4C_3^2\lambda_j^2}\Big)\le \exp\Big(-\frac{n(\lambda_s-\lambda_{s+1})^2}{4C_3^2\lambda_s^2}\Big),
\end{align*}
provided that 
\[
\frac{\lambda_j}{\lambda_j-\lambda_{s+1}}\sum_{k\le j}\frac{\lambda_k}{\lambda_k-\lambda_{s+1}}\le \frac{n}{4C_3^2}.
\]
The latter condition is strongest for $j=s$ and the resulting condition is implied by \eqref{ThmVarBoundCond}. Similarly, by \cite[Corollary 3.12]{RW17} we have for $j\in J$, 
\[
\mathbb{P}(\hat\lambda_j-\lambda_j>(\lambda_r-\lambda_j)/2)\le \exp\Big(-\frac{n(\lambda_r-\lambda_j)^2}{4C_3^2\lambda_r^2}\Big)\le \exp\Big(-\frac{n(\lambda_r-\lambda_{r+1})^2}{4C_3^2\lambda_r^2}\Big),
\]
provided that 
\[
\frac{\lambda_r}{\lambda_r-\lambda_{j}}\sum_{k\ge j}\frac{\lambda_k}{\lambda_r-\lambda_{k}}\le \frac{n}{4C_3^2}.
\]
The latter condition is strongest for $j=r+1$ and the resulting condition is implied by \eqref{ThmVarBoundCond}. The claim now follows from the above concentration inequalities and the union bound.
\end{proof}

\begin{proof}[End of proof of Theorem \ref{ThmVarBound}]  
By Lemma \ref{lem1} and the bound $\mathbb{E}\|P_j\Delta P_k\|_2^2\le 16\SGN^2$ (cf. \cite[Section 2.1]{RW17}), we get
\begin{equation*}
\mathbb{E}\sum_{k\notin J}g_k\|\hat P_JP_{k}\|_2^2\le 256\SGN^2\sum_{j\in J}\sum_{k\notin J}\frac{\lambda_j\lambda_k}{g_k}+\lambda_1(s-r)\mathbb{P}(\mathcal{E}^c)
\end{equation*}
and the first claim follows from Lemma \ref{Lem2}. Inserting \eqref{ThmVarBoundCond} into the first claim gives \eqref{ThmExpVarBound}.
\end{proof}

\subsection{Proof of Lemma \ref{LemApproxPolDecay1}} \label{ProofLemApproxPolDecay1}
In the proof, we will make use of the fact (cf. Equation (2.15) in \cite{RW17}) that there is a constant $C > 0$ depending only on $\alpha$ such that, for every $r \ge 1$,
\begin{equation}\label{EqEVEPD}
\sum_{j\le r}\frac{j^{-\alpha}}{j^{-\alpha}-(r+1)^{-\alpha}}\le Cr\log(er),\quad\sum_{k>r}\frac{k^{-\alpha}}{r^{-\alpha}-k^{-\alpha}}\le Cr\log(er).
\end{equation}
\begin{proof}[Proof of (a)]
Let $0<c_1<1$ be a constant such that $\CEV^{-2}(2c_1)^{-\alpha}>1$. By continuity, there is a constant $0<c_2<\CEV^{-1}$ such that $(\CEV^{-1}-c_2)(\CEV-c_2)^{-1}(2c_1)^{-\alpha}>1$. Let $d\ge 1/c_1$ and let $[c_1d, d]\cap\mathbb{N}=\{f,f+1\dots,d\}$. We first show that there is an $r\in\{f,f+1\dots,d\}$ such that
\begin{equation}\label{EqExistGap}
\forall j\le r,\qquad \lambda_j-\lambda_{r+1}> c_2(j^{-\alpha}-(r+1)^{-\alpha}).
\end{equation}
Arguing by contradiction, let us assume that for every $r\in\{f,f+1\dots,d\}$, \eqref{EqExistGap} does not hold. Then there exists a sequence of natural numbers $d+1=j_1>j_2>\dots>j_{M-1}>f\ge j_M$ with $\lambda_{j_{a+1}}-\lambda_{j_a}\le c_2(j_{a+1}^{-\alpha}-j_a^{-\alpha})$ for $a=1,\dots, M-1$. We get
\begin{equation*}
\lambda_{j_M}-\lambda_{j_1}=\sum_{a=1}^{M-1}(\lambda_{j_{a+1}}-\lambda_{j_a})\le c_2\sum_{a=1}^{M-1}( j_{a+1}^{-\alpha}-j_a^{-\alpha})=c_2(j_{M}^{-\alpha}-j_1^{-\alpha})
\end{equation*}
and thus $\CEV^{-1} j_M^{-\alpha}-\CEV j_1^{-\alpha}\le c_2(j_{M}^{-\alpha}-j_1^{-\alpha})$, by invoking also Assumption \ref{ApproxPolDecay}. Rearranging and using the bound $j_M\le f \le c_1d+1\le 2c_1d$, we have
\[
(\CEV^{-1}-c_2)(2c_1)^{-\alpha}d^{-\alpha}\le (\CEV^{-1}-c_2)j_M^{-\alpha}\le (\CEV-c_2)j_1^{-\alpha}\le (\CEV-c_2)d^{-\alpha},
\]
which contradicts the definition of $c_2$. Hence, there is an $r\in\{f,f+1\dots,d\}$ such that \eqref{EqExistGap} holds and for this choice, \eqref{EqExistGap}, Assumption \ref{ApproxPolDecay}, and \eqref{EqEVEPD} give
\[
\sum_{j\le r}\frac{\lambda_j}{\lambda_j-\lambda_{r+1}}\le \frac{\CEV}{c_2}\sum_{j\le r}\frac{j^{-\alpha}}{j^{-\alpha}-(r+1)^{-\alpha}}\le C_1r\log(er)
\]
with a constant $C_1>0$ depending only on $\CEV$ and $\alpha$. The second claim follows similarly.
\end{proof}

\begin{proof}[Proof of (b)]
Similarly as in the proof of Lemma \ref{LemApproxPolDecay1}, let $C_1>1$ and $0<c_1<\CEV^{-1}/2$ be constants such that $\CEV^{-1}-2c_1>(\CEV-2c_1)C_1^{-\alpha}$. Let $d\ge 1$ and let $[d, C_1d]\cap\mathbb{N}=\{d,d+1\dots,f\}$. Let 
\[
I=\{r\in\{d,\dots,f\}:\forall j\le r, \lambda_j-\lambda_{r+1}> c_1(j^{-\alpha}-(r+1)^{-\alpha})\},
\] 
and let 
\[
J=\{s\in\{d,\dots,f\}:\forall k>s, \lambda_s-\lambda_k> c_1(s^{-\alpha}-k^{-\alpha})\}.
\] 
By definition of $I$, there exists $M\ge 0$ and a sequence of natural numbers $r_1\ge j_1> r_2\ge j_2> \dots> r_{M}\ge j_{M}$ with $r_1=\max\{d,\dots,f\}\setminus I\le f$ and $r_{M}\ge d$ such that 
\begin{equation*}
 \{d,\dots,f\}\setminus I \subseteq \bigcup_{a=1}^{M}\{j_{a},\dots,r_a\},\quad \lambda_{j_a}-\lambda_{r_a+1}\le c_1(j_a^{-\alpha}-(r_{a}+1)^{-\alpha})\ \forall a\le M.
\end{equation*}
Similarly, there exists $N\ge 0$ and a sequence of natural numbers $s_1<k_1\le s_2<k_2\le \dots\le s_{N}<k_{N}$ with $s_1=\min\{d,\dots,f\}\setminus J\ge d$ and $s_{N}\le f$ such that
\begin{equation*}
\{d,\dots,f\}\setminus J \subseteq \bigcup_{b=1}^{N}\{s_b,\dots,k_b-1\},\quad\lambda_{s_b}-\lambda_{k_{b}}\le c_1({s_b}^{-\alpha}-{k_b}^{-\alpha})\ \forall b\le N.
\end{equation*}
We now show that $I\cap J\neq \emptyset$. We argue by contradiction, and assume that $I\cap J=\emptyset$. Then we have the covering
\begin{align*}
\{d',\dots,f'\}&=\{d,\dots,f\}\cup \{j_M,\dots,r_M\}\cup\{s_N,\dots,k_{N}-1\}\\
&=\bigcup_{a=1}^{M}\{j_{a}\dots,r_a\}\cup\bigcup_{b=1}^{N}\{s_b,\dots,k_b-1\}
\end{align*}
with $d'=\min(d, j_{M})$ and $f'=\max(f,k_{N}-1)$. We get
\begin{align*}
\lambda_{d'}-\lambda_{f'+1}&\le \sum_{a=1}^{M}(\lambda_{j_a}-\lambda_{r_a+1})+\sum_{b=1}^N(\lambda_{s_b}-\lambda_{k_{b}})\\
&\le c_1\sum_{a=1}^{M}(j_a^{-\alpha}-(r_{a}+1)^{-\alpha})+c_1\sum_{b=1}^N(s_b^{-\alpha}-k_b^{-\alpha})\\
&\le 2c_1({d'}^{-\alpha}-(f'+1)^{-\alpha})
\end{align*}
and thus $\CEV^{-1}{d'}^{-\alpha}-\CEV(f'+1)^{-\alpha}\le 2c_1({d'}^{-\alpha}-(f'+1)^{-\alpha})$, by invoking also Assumption~\ref{ApproxPolDecay}. Rearranging and using the bound $f'+1\ge C_1d$, we have
\[
(\CEV^{-1}-2c_1)d^{-\alpha}\le (\CEV^{-1}-2c_1){d'}^{-\alpha}\le (\CEV-2c_1)(f'+1)^{-\alpha}\le (\CEV-2c_1)C_1^{-\alpha}d^{-\alpha},
\]
which contradicts the definition of $c_1$. Hence, $I\cap J\neq \emptyset$ and the claim follows similarly as in the proof of Lemma \ref{LemApproxPolDecay1} from applying the definitions of $I$ and $J$, Assumption \ref{ApproxPolDecay}, and \eqref{EqEVEPD} to $r\in I\cap J$.
\end{proof}

\end{document}